\documentclass[11pt]{amsart}
\usepackage{amsfonts}

\usepackage[all,cmtip]{xy}
\usepackage{amsmath}
\usepackage{amsthm}
\usepackage{amssymb}
\usepackage{enumitem}
\newtheorem{thm}{Theorem}[section]
\newtheorem{conj}{Conjecture}[section]

\newtheorem*{definition*}         {Definition}
\newtheorem{lemma}[thm]{Lemma}

\newtheorem{cor}[thm]{Corollary}

\theoremstyle{remark}

\newcommand*{\Q}{\mathbb{Q}}
\newcommand*{\Hh}{\mathbb{H}}
\newcommand*{\Qa}{\overline{\mathbb{Q}}}
\newcommand*{\Z}{\mathbb{Z}}
\newcommand*{\G}{\mathbb{G}}
\newcommand*{\A}{\mathcal{A}}
\newcommand*{\R}{\mathbb{R}}

\newcommand*{\C}{\mathbb{C}}

\newcommand*{\DT}{\mathbb{S}}

\newcommand*{\Disc}{\textrm{Disc}}

\def\Sp{{\rm Sp}}

\usepackage{amscd,amssymb,amsmath}
\usepackage{color}

\author{Jonathan Pila and Jacob Tsimerman
}
\address{Mathematical Institute, University of Oxford, UK}
\address{Department of Mathematics, Harvard University, USA}

\begin{document}

\centerline{Accepted for publication in the {\it Annals of Mathematics}}

\bigskip

\begin{abstract}
We prove the Ax-Lindemann theorem for the coarse moduli space
$\mathcal{A}_{g}$ of principally polarized abelian varieties of
dimension $g\ge 1$. We affirm the Andr\'e-Oort conjecture
unconditionally for $\A_g$ for
$g\le 6$, and under GRH for all $g$.
\end{abstract}

\title{Ax-Lindemann For $\A_g$}

\maketitle

\section{introduction}

In this paper we prove the ``Ax-Lindemann'' theorem for
$\A_g=\A_{g,1}, g\ge 1$,
the moduli space of principally polarized
Abelian varieties of dimension $g$. The statement of the theorem is as follows
(for definitions and conventions see \S2).
Let $\Hh_g$ be the Siegel upper-half space and
$\pi_{g}: \Hh_g\rightarrow\A_g$
the ${\rm Sp}_{2g}(\Z)$-invariant uniformisation.

\begin{thm}\label{1.1}
Let $V\subset \A_g$ be a subvariety and
$W\subset\pi_{g}^{-1}(V)$  a maximal algebraic subvariety.
Then $W$ is weakly special.
\end{thm}

As explained in \cite{PILAOAO}, this theorem may be viewed
as an analogue for the map $\pi_g$ of
part of Ax's theorem \cite{AX} establishing the differential field version
of Schanuel's
conjecture for the exponential function  (see \cite{LANG}, p30),
namely the part
which corresponds to the classical
Lindemann (or Lindemann-Weierstrass) theorem (see \cite{LANG}, p77).
Hence the neologism ``Ax-Lindemann''.
The proof of Theorem 1.1 combines various arithmetic
estimates with the Counting Theorem of Pila-Wilkie \cite{PW}
and with the idea of Ullmo-Yafaev \cite{UYHYPERBOLIC}
to use hyperbolic volume at the boundary.

The Andr\'e-Oort conjecture (AO) is a compositum of conjectures made
by Andr\'e \cite{ANDRE} and Oort \cite{OORT}. A full proof of AO
under the assumption of GRH for CM fields
has been announced by Klingler, Ullmo, and Yafaev
\cite{KYAO, UYAO}.
An appropriate Ax-Lindemann theorem
is a key ingredient in proving cases of AO
unconditionally using o-minimality and point-counting
\cite{DAWYAFAEV, PILAOAO, PT, ULLMO},
following the basic strategy originally proposed by Zannier for
re-proving the Manin-Mumford conjecture \cite{PZ}. It provides a
geometric characterisation of the exceptional set in the
Counting Theorem, and in this role is analogous to functional
transcendence statements in the context
of other Zilber-Pink type problems (e.g. \cite{MZ, HP}).\/
The following theorem affirms AO unconditionally for $\A_g$ for
$g\le 6$, and for all $g$ under the assumption of GRH.

\begin{thm}\label{1.2}
Let $V\subset\A_g, g\le 6$ be a subvariety. Then $V$ contains only finitely many maximal special subvarieties.
Under the assumption of GRH (for CM fields)
the same conclusion holds for all $g$.
\end{thm}

Both theorems rely on the definability in the o-minimal structure
$\R_{\rm an,\ exp}$ of the map $\pi_{g}:\Hh_g\rightarrow\A_g$
when restricted to a standard fundamental
domain for the ${\rm Sp}_{2g}(\Z)$ action.
This result, stated in \S2.7, is due to Peterzil-Starchenko \cite{PEST}.
The o-minimality of $\R_{\rm an,\ exp}$ is due to
van den Dries and Miller \cite{DM}, building on
the fundamental work of Wilkie \cite{WILKIE}; for further references
on o-minimality see \S2.

The restriction to $g\le 6$ for the unconditional statement in
1.2 is due to another ingredient which is
crucial to the strategy: a suitable lower bound for the size of the
Galois orbit of  a special point.
These have been established by the second author \cite{JT}
unconditionally for $g\le 6$, and for all $g$ on GRH
(the former for $g\le 3$ and the latter were shown
independently by Ullmo-Yafaev \cite{UYGB}).
However, we show that such bounds are the only remaining obstacle
to proving AO for $\A_g$ in general. This gives a new proof of AO for
$ \A_g$ for all $g$ assuming GRH, by different methods to the ones
employed by Klingler-Ullmo-Yafaev \cite{KYAO, UYAO}.

In the course of preparation of this manuscript, the
preprint \cite{ULLMO} by Ullmo  appeared showing how to deduce AO
from these ingredients (Ax-Lindemann, height upper bound for
preimages of special points, Galois lower bounds, and
uniformisation with suitable definability)
for any Shimura variety. Ullmo thereby proves (in \cite{ULLMO}) AO unconditionally
for all projective Shimura subvarieties of $\A_6^n, n\ge 1$, using the
Ax-Lindemann theorem for projective Shimura varieties established
by Ullmo-Yafaev \cite{UYHYPERBOLIC}, the height upper bounds in \cite{PT},
and the Galois lower bounds of \cite{JT},
as the definability requirements are easily seen to be satisfied (in $\R_{\rm an}$)
in the case of a projective Shimura variety.
The deductions of AO in \cite{ULLMO} and here
differ in detail but both depend on the structure of
weakly special subvarieties of $\A_g$ and make further crucial use of o-minimality
(as in \cite{PILAOAO, PT}).
We have retained our treatment in order to keep our paper self-contained.

We begin in \S2 by reviewing $\Hh_g$ and our basic definitions and conventions.
Then we prove some norm estimates in \S3,
and estimates about volumes of curves in fundamental domains (\S4)
and near the boundary of $\Hh_g$  (\S5).
With these preparations we prove 1.1 in \S6.
The deduction of AO from 1.1 and the various
other ``ingredients'' is carried out in \S7.

\bigskip
\noindent
{\bf Acknowledgements.\/} We thank Martin Orr, Kobi Peterzil, and
Sergei Starchenko for their comments. We are extremely
grateful to  the referees for their thorough reading of earlier
versions of this paper and their detailed comments and corrections.

\section{Basic conventions and notation}

\subsection{$\Sp_{2g}(\R)$ and $\Hh_g$}

The {\it symplectic group\/} $\Sp_{2g}(R)$ with entries in a ring $R$ is
the group of matrices $T\in M_{2g}(R)$ satisfying
$$
TJT^t=J
$$
where $J=\begin{pmatrix}0& 1_g\cr -1_g&0\end{pmatrix}$ is the standard
alternating matrix of degree
$2g$, and $T^t$ is the transpose of $T$; see e.g. \cite{IGUSA}. If we write
$$
T=\begin{pmatrix} A & B \\ C & D\end{pmatrix},\quad
A,B,C,D\in M_g(R)
$$
then the condition $TJT^t=J$ is equivalent (see \cite{GEER}, p183) to
$$
AB^t=BA^t,\quad CD^t=DC^t,\quad AD^t-BC^t=I_g.
$$
We know (see \cite{GEER}, p184) that
$$\begin{pmatrix} A & B \\ C & D\end{pmatrix}\in \Sp_{2g}(R)\Rightarrow
A^tC=C^tA, B^tD=D^tB,$$  so that $\Sp_{2g}(R)$ is closed under transposition.

The {\it Siegel upper half-space $\Hh_g$\/} is defined to be
$$\Hh_g = \left\{Z\in M_g(\C): Z=Z^t, {\rm Im}(Z)>0 \right\}.$$
Thus $\Hh_g$ is an open domain in the space
$M_{g}(\C)^{\rm sym}$ of symmetric
complex $g\times g$ matrices, which we may identify with $\C^{g(g+1)/2}$.
There is an action of $\Sp_{2g}(\R)$ on $\Hh_g$ given by
$$\begin{pmatrix} A & B \\ C & D\end{pmatrix}Z = (AZ+B)(CZ+D)^{-1}.$$
Denote by $\A_g$ the coarse moduli space of complex principally
polarized abelian varieties of dimension $g$.
The quotient $\Sp_{2g}(\Z)\backslash\Hh_g$ is isomorphic to $\A_g$.
We write $\pi_g:\Hh_g\rightarrow\A_g$ for the projection map, which
may be given explicitly by scalar Siegel modular forms of suitable weight;
see \cite{GEER}, \S10 and \S11.

\subsection{Varieties and subvarieties}

We identify varieties with their sets of complex-valued points. By a
{\it subvariety\/} $V$ of a quasiprojective variety\/ $W$ we mean
a Zariski closed subset. An {\it irreducible algebraic subvariety\/}
of $\Hh_g$, considered as a subset of $\C^{g(g+1)/2}$,
means an irreducible
component  (as a complex analytic variety) of $\Hh_g\cap Y$
where $Y\subset\C^{g(g+1)/2}$ is a subvariety.
Likewise, if $\Hh$ is any hermitian domain, then there is a
natural algebraic variety $\Hh_0$ of which
$\Hh$ is an open subset (see \cite{UYS}), and we define
an irreducible algebraic subvariety of $\Hh$ to be an
irreducible component, as a complex analytic variety,
of $\Hh\cap Y$, where $Y$ is an algebraic subvariety
of $\Hh_0$. For $A\subset \Hh_g$, a
{\it maximal algebraic subvariety\/} of $A$
is an irreducible algebraic subvariety $W\subset\Hh_g$ with
$W\subset A$
such that if $W'\subset\Hh_g$ is an irreducible algebraic subvariety with
$W\subset W'\subset A$ then $W=W'$.

\subsection{The metric}

There is an invariant metric on $\Hh_g$ for the action of $\Sp_{2g}(\R)$,
and this is given (see e.g. \cite{SIEGEL}, p17) by

$$d\mu(Z)={\rm Tr\/} (Y^{-1}dZY^{-1}d\bar{Z}).$$

\subsection{Shimura data and special subvarieties}

We gather here some basic facts about Shimura varieties;
for a more thorough account with proofs see \cite{D1} or \cite{D2}.
We work exclusively with
connected Shimura varieties, of which
an excellent account can be found in \cite{P}.

Define the Deligne torus $\DT$ be the real torus given by Weil
restricting $\G_m$ from $\C$ to $\R$. Thus, the real points
$\DT(\R)$ can be identified with $\C^{\times}$.
A \emph{connected Shimura datum\/}
is defined to be a pair $(G,X^{+})$ where $G$ is a reductive
group over $\Q$  and $X^{+}$ is a connected component of a $G(\R)$-conjugacy class $X$ of homomorphisms
$h:\DT\rightarrow G_{\R}$ satisfying the following conditions:
\begin{itemize}
\item For all $h\in X$, only the weights $(0,0),(1,-1),(-1,1)$
may occur in the
adjoint action of $\DT$ on the complexified Lie algebra of $G$
\item ${\rm ad}\circ h(i)$ is a Cartan involution on the adjoint group
$G^{\rm ad}_{\R}$    and
\item $G^{\rm ad}$ has no simple $\Q$-factor $H$ such that
$H(\R)$ is compact.
\end{itemize}

These axioms ensure that if we set $K_{\infty}$ to be the
stabilizer of some
$h\in X$, then $X=G(\R)\slash K_{\infty}$ is a finite
union of hermitian symmetric
domains on which $G(\R)$ acts transitively and biholomorphically.
We say that $(H,X_H)$ is a \emph{Shimura subdatum} of $(G,X_G)$ if $H\subset G$ and
$X_H\subset X_G$.
A \emph{connected Shimura variety} is defined to be
$\Gamma\backslash X^{+}$, where
$\Gamma\subset G(\Q)$ is a congruence subgroup of $G(\R)$ that
stabilises $X^+$ (\cite{P}).
Thus, to give a connected Shimura variety is the same as to give a triple
$(G,X^+,\Gamma)$.
A connected Shimura variety can be given the structure
of an algebraic variety defined over a number field.

We denote by $G^{\rm ad}$ the adjoint form of a reductive group $G$. Given a Shimura datum $(G,X_G)$ we get an associated
Shimura datum $(G^{\rm ad},X_{G^{\rm ad}})$ such that
we can identify $X_G^+$ with $X^+_{G^{\rm ad}}$.

A {\it morphism of connected Shimura varieties\/}
\[\phi:(H,X_H^+,\Gamma_H)\rightarrow (G,X_G^+,\Gamma_G)\]
is a morphism of $\Q$-groups $\phi:H\rightarrow G$ which carries $X_H^+$
to $X_G^+$ and $\Gamma_H$ to $\Gamma_G$. The image
$\phi(\Gamma_H\backslash X_H^+)$ is an algebraic subvariety of
$\Gamma_G\backslash X_G^+$.

\begin{definition*}

 A subvariety $V\subset \Gamma_G\backslash X_G^+$ is called a \emph{special subvariety} if there
exists a morphism of connected Shimura varieties $\phi:(H,X_H^+,\Gamma_H)\rightarrow (G,X_G^+,\Gamma_G)$ such that
$\phi(\Gamma_H\backslash X_H^+)=V$.
By abuse of notation, we also refer to
$\phi(X_H^+)$ as a special subvariety of $X_G^+$. Keep in mind that
$\phi(X_H^+)$ is an irreducible algebraic subvariety of $X_G^+$
under its semi-algebraic structure as a hermitian symmetric space.
A \emph{special point\/}  is a special subvariety of dimension zero
(in $\Gamma_G\backslash X_G^+$ or in $X_G^+$).

\end{definition*}

For our purposes, we shall need a slightly more general notion
(see \cite{UYS}).

\begin{definition*}
We say that $V\subset \Gamma_G\backslash X_G^+$  is a \emph{weakly special subvariety} if there exist Shimura varieties
$(H_i,X_i^+,\Gamma_i)$ for $i=1,2$ and a Shimura subdatum $(H,X^+_H)\subset (G,X^+_G)$ such that
$$
(H^{\rm ad},X^+_{H^{\rm ad}})=(H_1,X^+_1)\times(H_2,X^+_2)
$$
and $y\in \Gamma_2\backslash X_2^{+}$ such that
$V$ is the image of $X^+_1\times\{y\}$ in $\Gamma_G\backslash X^+_G.$
By the same abuse of notation as above, we refer to the image of $X_1^+\times\{y\}$ in $X^+_G$ as a weakly special subvariety of
$X_G^+$.

\end{definition*}

A weakly special subvariety is special if and only if it contains
a special point; see \cite{MOONEN}.

The following result of Ullmo-Yafaev
is very useful in working with weakly special subvarieties:

\begin{lemma}\cite{UYS}\label{UYS}
An irreducible algebraic variety $S\subset\Hh_g$ is weakly special
iff $\pi_g(S)$ is algebraic.
\end{lemma}

\begin{cor}\label{INT}
An irreducible component of the intersection of weakly special subvarieties
is weakly special.
\end{cor}

\begin{lemma}\label{WeakChar}
Let $(G,X_G^+)$ be a connected Shimura datum.
Let $V$ be a weakly special subvariety of $X_G^+$. Then there exists
a semi-simple subgroup $G_0\subset G$ defined over $\Q$,
with no compact $\Q$-factors, and a point
$Z_0\in X_G^+$ such that the image of $G_{0,\R}$ in
$G^{\rm ad}_{\R}$ is
fixed by the Cartan involution
corresponding to $Z_0$ and $V=G_0(\R)^+\cdot Z_0$.
\end{lemma}
\begin{proof}

Since $V$ is weakly special, there are connected
Shimura data 
$$(H^{\rm ad},X^{+}_{H^{\rm ad}})=(H_1,X_1^+)\times(H_2,X_2^+)$$
where $(H,X_H)$ is a Shimura subdatum of $(G,X_G)$,
and a point $y\in X^+_2$
such that $V$ is the image of $\phi(X_1^+\times\{y\})$ in 
$\Gamma_G\backslash X^+_G$.

Pick $h\in X_1^+$ and set
$Z_0=(h,y)$, where as always we identify 
$X^{+}_{H^{\rm ad}}$ with $X^+_H$.
Finally, take $G_1\subset H$ to be $\phi$ composed 
with the pullback under the map $H\rightarrow H^{\rm ad}$
of $H_1\times\{1_{H_2}\}$.
Then the Cartan involution corresponding to $Z_0$
is given by conjugation
by a lift to $H(\R)$ of $(h,y)(i)$. Thus, $G_{1,\R}$ is fixed 
under this involution.
By definition, we have
$$V=X_1^+\times\{y\}= G_1(\R)^+\cdot Z_0,$$ as desired. 
Finally, we let $G_0$ be the derived subgroup
of $G_1$. The fact that $G_0^{\rm ad}$ has no compact $\Q$-factors
follows from the fact that $H_1^{\rm ad}$ has none.
\end{proof}

Now, given an irreducible variety $V\subset\A_g$, let $S(V)$
be the minimal weakly special subvariety containing
$V$. Let $V^{\rm sm}$ be the smooth locus of $V$, and  let $U_0$
be a connected component of $\pi_g^{-1}(V^{\rm sm})$. Define the
{\it monodromy group\/} $\Gamma_V\subset \Sp_{2g}(\Z)$
to be the stabilizer of $U_0$. Let $G_0$ be the identity
component of the Zariski closure of $\Gamma_V$.

\begin{lemma}\label{MinimalWS}
The group $G_0$ is semi-simple and
$\pi_g(G_0(\R)^+\cdot v)= S(V)$ for all $v\in U_0$.
\end{lemma}
\begin{proof}

The semi-simplicity would follow from Theorem 1 of \cite{A}, 
except for the fact that $\Sp_{2g}(\Z)$ does not act freely on 
$\Hh_g$ and thus $\A_g$ does not possess
a family of polarized Hodge structures. However, consider the full 
congruence group $\Gamma(3)\subset \Sp_{2g}(\Z)$. 
This subgroup does act freely on $\Hh_g$, so that
$\A_{g,3}:=\Gamma(3)\backslash\Hh_g$ does have a family of 
polarized Hodge structures. Now, we have the finite map 
$\phi_3:\A_{g,3}\rightarrow\A_g$. Let $V_3\subset \A_{g,3}$ 
be an irreducible component of $\phi_3^*V$. We can apply 
Theorem 1 of \cite{A} to $V_3$ and get that the identity component 
of the Zariski closure of $\Gamma_{V_3}$ is semi-simple. 
However, $\Gamma_{V_3}$ is a finite index normal subgroup of 
$\Gamma_V$ and thus their Zariski closures have the same 
identity component, proving the semi-simplicity of $G_0$.
The rest of the lemma follows  from \cite{MOONEN}, 
Lemmas 3.7 and 3.9
\end{proof}

\subsection{Fundamental domains}

We set $S_g$ to be the usual Siegel fundamental domain for
the action of
$\Sp_{2g}(\Z)$ on $\Hh_g$ (see \cite{PEST}, or \cite{IGUSA}, \S 3.3).

\begin{lemma}\label{fundomain}
Suppose $X+iY\in S_g$. Then there exists a constant 
$c_g>0$ such that
\begin{enumerate}[label=(\alph*)]
\item All the coefficients of $X$ are bounded in 
absolute value by $\frac{1}{2}$
\item $|Y|\leq\prod_{y=1}^g y_{ii} \leq c_g|Y|$, 
where $|Y|$ is the determinant of $Y$
\item $\frac{\sqrt{3}}{2}\leq y_{11}\leq 
y_{2 2}\leq y_{33}\dots\leq y_{gg}$
\item $\forall 1\leq i\neq j\leq g,  |y_{ij}|\leq \frac12\min(y_{ii},y_{jj})$.
\end{enumerate}
\end{lemma}
\begin{proof}
Parts (a) and (d) are true by definition, while (c) follows
from \cite{IGUSA}, Lemma 15, p.195.
Part (b) follows from Minkowski's 2nd result on successive minima,
together with \cite{K}, Theorem 2.
\end{proof}

\subsection{Norms}

We define the {\it norm\/} of a matrix
$Z=(z_{ij})=X+iY\in \Hh_g$, where $X,Y\in M_g(\R)$,  to be
$$h(Z)={\rm max\/} (1,|z_{ij}|, |Y|^{-1}),$$
and we define the {\it norm\/} of a matrix
$M=(m_{ij})\in {\rm M\/}_{n}(\R)$ to be
$$h(M)={\rm max\/}(1,|m_{ij}|).$$

\subsection{O-minimality and definability}

In this paper {\it definable\/} will mean {\it definable
in the o-minimal structure $\R_{\rm an\ exp}$.\/}
See \cite{PILAOAO} for an introduction o-minimality and
\cite{DM, DMM} for the properties of $\R_{\rm an\ exp}$.
An essential input to enable the o-minimal machinery to be applied
is the following result due to Peterzil-Starchenko \cite{PEST}:

\begin{thm}\label{define}
The projection map $\pi_g:S_g\rightarrow\A_g$ is definable.
\end{thm}

\section{Some norm bounds}

In this section we prove some basic lemmas concerning the norms
introduced in 2.6 above.  We shall only care
about asymptotic growth, and moreover we only wish to
establish that certain
quantities do not exhibit super-polynomial growth.
Thus, we introduce some notation.

\begin{definition*}
Let $M$ be a set and let $F, G$ be functions mapping $M$ to $\R_{>0}$.
We say that $F$ is \emph{polynomially bounded} in $G$ if there exist constants
$a,b>0$ with $F(m)\leq a\cdot G(m)^b$, and write $F\prec G$.
If $F\prec G$ and $G\prec F$ we write $F\asymp G$.
\end{definition*}

Clearly, $\prec$ is transitive, whereas $\asymp$ is an equivalence relation.
It is also clear that $F\prec G\Leftrightarrow G^{-1}\prec F^{-1}$.
We record some basic facts.

\begin{lemma}\label{basicheight}
If $Z\in\Hh_g$ and $M_1, M_2\in\Sp_{2g}(\R)$ we have:
\begin{enumerate}
\item $h(M_1M_2)\prec h(M_1)h(M_2)$
\item $h(M_1)\asymp h(M_1^{-1})$
\item $h(M_1Z)\prec h(M_1)h(Z)$.

\end{enumerate}
\end{lemma}
\begin{proof}

\begin{enumerate}

\item Clear.
\item By symmetry, it is enough to show $h(M_1^{-1})\prec h(M_1)$.
This is obvious because the minors of $M_1$ are polynomial in the
entries of $M_1$, and $|M_1|=1$.
\item Write $M_1=\begin{pmatrix} A & B\\ C & D\end{pmatrix}$.
Then $M_1Z = (AZ+B)(CZ+D)^{-1}$. We first bound $|CZ+D|$ from below.
Note the identity (\cite{GEER}, p184)
$$({}^*)\quad\quad\quad
(C\overline{Z}+D)^t(AZ+B)-(A\overline{Z}+B)^t(CZ+D)=2iY.$$

Let $f=|CZ+D|$. Then there exists a vector $\xi\in\C^g$ with $||\xi||=1$ and
$||(CZ+D)\xi||\prec f,$ where $||\xi||$ is defined to be $\bar{\xi}^t\cdot\xi$.

Complete $\xi$ to a unitary basis
$\xi=\xi_1,\xi_2,\dots,\xi_g$. Then we have
\[|Y|\leq \prod_{i=1}^g|\xi_i^tY\bar{\xi_i}|\prec |\xi^tY\bar{\xi}|h(Y)\] so that
$$|\xi^tY\bar{\xi}|^{-1}\prec h(Y)\leq h(Z).$$

Finally, (*) gives
$$
|\xi^tY\bar{\xi}| \prec fh(M_1)h(Z),
$$
so that $f^{-1}\prec h(M_1)h(Z)$. It is now clear that all the entries of
$M_1Z$ are polynomially bounded by $h(M_1)h(Z)$. It remains to
show that if $M_1Z= X'+iY'$ then $|Y'|^{-1}\prec h(M_1)h(Z)$.

Again using (*) and the fact that $M_1Z$ is symmetric, we derive
$(Y')^{-1} = (C\overline{Z}+D)Y^{-1}(CZ+D)^t.$
Thus all the coefficients of  $(Y')^{-1}$ are
polynomially bounded by $h(M_1)h(Z)$, and thus so is $|Y'|^{-1}$.
This completes the proof.
\end{enumerate}
\end{proof}

We restate here Lemma 3.2 from \cite{PT}. For a point
$Z\in\Hh_g$ there is a unique $\gamma_Z\in\Sp_{2g}(\Z)$ such that
$\gamma_Z\cdot Z\in S_g$.

\begin{lemma}\label{movingbound}
$h(\gamma_Z)\prec h(Z)$.
\end{lemma}

\section{Volumes of algebraic curves in $\Hh_g$.}

\subsection{Volumes of algebraic curves in $\Hh_g$.}

Recall that there is a complex structure on $\Hh_g$ via the
imbedding $\Hh_g\hookrightarrow M_g(\C)^{\rm sym}=\C^{g(g+1)/2}$.
Take $C\subset\Hh_g$ to be a curve in $\Hh_g$ (i.e. an irreducible algebraic subvariety of dimension 1).
We define the degree of $C$ to be the degree of the Zariski closure of $C$
in $\C^{g(g+1)/2}$. The restriction
of the metric $d\mu(Z)$ gives a Riemannian metric on $C$, and thus an induced
Volume form $dC$. Our goal in this section is to prove the following theorem.

\begin{thm}\label{SmallCurves}
For a curve $C\subset \Hh_g$ of degree $k$ we have the bound
\[\int_{C\cap S_g} dC \ll k\]
where the implied constant depends on $g$.
\end{thm}
\begin{proof}

The main ingredient in the proof of the theorem is the following lemma.

\begin{lemma}

Within $S_g$, we have $$d\mu(Z)=Tr(Y^{-1}dZY^{-1}d\overline{Z})\leq O_g(1)\cdot
\sum_{i,j}\frac{|dz_{ij}|^2}{y_{ii}y_{jj}}$$

\end{lemma}

\begin{proof}

By Lemma \ref{fundomain} we have that $|Y|\gg \prod_{i=1}^g y_{ii}$,
and $$\forall{i\neq j},|y_{ij}|\leq {\rm min}(y_{ii},y_{jj})/2.\leqno{(*)}$$

Now, let $y'_{ij}$ denote the entries of $Y^{-1}$, and let $M_{ij}$ be the $(i,j)$'th
minor of $Y$. Then by (*) and the expansion of $M_{i,j}$ along columns we see that
$|M_{ij}|\ll \prod_{l\neq j} y_{ll}$, and thus
$|y'_{ij}|=|Y|^{-1}|M_{ij}|\ll y_{jj}^{-1}$. Likewise $|y'_{ij}|\ll y_{ii}^{-1}$.
Thus $y'_{ij}\leq |y_{ii}\cdot y_{jj}|^{-\frac12}$.
Thus we have
\begin{align*}
Tr(Y^{-1}dZY^{-1}d\overline{Z}) &= \sum_{i,j,m,l}y'_{ij}dz_{jm}y'_{ml}d\overline{z_{li}}\\
&\ll \sum_{i,j,m,l}|dz_{jm}dz_{li}|(y'_{ii}y'_{jj}y'_{mm}y'_{ll})^{-\frac12}\\
&\leq \sum_{i,j,m,l} \frac{|dz_{jm}|^2}{y_{jj}y_{mm}} + \frac{|dz_{li}|^2}{y_{ii}y_{ll}}
\end{align*}
which is what we wanted to show.
\end{proof}

Since the volume form for the conformal metric $|dz_{ij}|^2$ is $dx_{ij}dy_{ij}$,
by the lemma above, it is enough to show that for all $(i,j)$ we have
$$\int_{C\cap S_g} \frac{dx_{ij}dy_{ij}}{y_{ii}y_{jj}} \ll k.$$
We first consider the case $i=j$. In this case, consider the projection map
onto the $z_{ii}$ coordinate, $\pi_{ii}:\Hh_g\rightarrow\Hh_1.$
By Lemma \ref{fundomain} the image of $S_g$ under $\pi_{ii}$ is contained
in the Siegel set $$y_{ii}>\frac{\sqrt{3}}{2}, |x_{ii}|\leq\frac12.$$
Moreover, when the map $\pi_{ii}$ is restricted to $C$ it is either constant,
in which case the differential $dz_{ii}$ vanishes along $C$, or it has
finite fibers. In fact, since $C$ has degree $k$, the map
$\pi_{ii}:C\rightarrow \Hh_1$ is at most $k$ to $1$.
Thus we have
\[\int_{C\cap S_g} \frac{dx_{ii}dy_{ii}}{y_{ii}^2}\leq
 k\int_{y_{ii}=\sqrt{3}/2}^{\infty}\int_{x_{ii}=-\frac12}^{\frac12}\frac{dx_{ii}dy_{ii}
 }{y_{ii}^2}=\frac{2k}{\sqrt{3}}\]

We now consider the case of $i\neq j$. Here we use that since
$y_{ii},y_{jj}>\frac{\sqrt{3}}{2}$, we have $\frac{1}{y_{ii}y_{jj}}\leq
{\rm min}(\frac1{y_{ij}^2},\frac43)$.
Thus projecting to the $z_{ij}$ coordinate and reasoning as before, we get
\[\int_{C\cap S_g} \frac{dx_{ij}dy_{ij}}{y_{ii}y_{jj}}\leq
k\int_{y_{ij}=-\infty}^{\infty}\int_{x_{ij}=-\frac12}^{\frac12}
{\rm min}(\frac43,\frac{1}{y_{ij}^2})dx_{ij}dy_{ij}\ll k\] as desired.
\end{proof}

\subsection{Volumes of algebraic curves in weakly special subvarieties.}\label{curvsubv}

Let $G$ be a connected semi-simple algebraic subgroup of $\Sp_{2g}$ defined
over $\Q$, with no simple compact $\Q$-factors,  and $Z_0\in\Hh_g$ such that $\Hh_G:=G(\R)^+\cdot Z_0$ is a
weakly special subvariety with $G_0^{\rm ad}(\R)$ fixed by the
Cartan involution corresponding to $Z_0$. Our goal is to prove
a similar statement to
Theorem \ref{SmallCurves} with $\Hh_G$  replacing $\Hh_g$.

Let $V_G:=\pi_g(\Hh_G)$ be the weakly special subvariety of $\A_g$,
and let $$\Gamma_G\subset {\rm Sp}_{2g}(\Z)$$ be the lattice
which stabilizes $\Hh_G$. Necessarily, $\Gamma_G$ contains
$G(\Z)^+$, the intersection of $G(\R)^+$ with ${\rm Sp}_{2g}(\Z)$,
as a
subgroup of finite index. Note that $\pi_g^{-1}(V_G)$ is the (not disjoint)
union over $\gamma\in\Sp_{2g}(\Z)$
of $\gamma\Hh_G$. Now, consider $\pi_g^{-1}(V_G)\cap S_g$.
By Theorem \ref{define} this intersection is definable and thus has
only finitely many connected components.
Thus, there are finitely many elements
$\gamma_i\in\Sp_{2g}(\Z),1\leq i\leq m$ such that
$$\pi_g\left(\bigcup_{i=1}^m \gamma_i^{-1}\Hh_G\cap S_g\right)=V_G$$
and thus
$$\pi_g\left(\bigcup_{i=1}^m\Hh_G\cap \gamma_i S_g\right)=V_G.$$
\begin{lemma}
With notation as above,
$S_G^0:=\bigcup_{i=1}^m \gamma_i S_g \cap \Hh_G$
contains
a fundamental domain for the action of $\Gamma_G$.
\end{lemma}

\begin{proof}

Indeed, let $Z\in \Hh_G$. Then as $S_g$ is a fundamental domain for
$\Sp_{2g}(\Z)$, there exists $\gamma\in\Sp_{2g}(\Z)$ such that
$\gamma\cdot Z\in S_g$. Now, as above this means that there is an
$i$ with $1\leq i\leq m$ such that
$\gamma\cdot Z\in \gamma_i^{-1} \cdot\Hh_G$ and
$\gamma\cdot \Hh_G=\gamma_i^{-1}\cdot \Hh_G$.
Thus there exists a $g\in\Gamma_G$ with
$\gamma=\gamma_i^{-1}\cdot g$. Thus
$\gamma_i^{-1}gZ\in S_g\cap \gamma_i^{-1}\Hh_G$ and
so $gZ\in \gamma_i S_g\cap \Hh_G\subset S_G^0$ as desired.
\end{proof}

Now, by picking coset representatives for $G(\Z)^+$ in $\Gamma_G$
we can find a finite union
of elements $\beta_j\in \Sp_{2g}(\Z)$ such that $S_G:=\cup_j \beta_jS_G^0\cap \Hh_G$
contains a fundamental domain for the action of $G(\Z)^+$ on $\Hh_G$.

Summarizing the above discussion, we arrive at the following theorem:

\begin{thm}\label{SmallCurvesGen}
There is a semi-algebraic set $S_G\subset\Hh_G$ such that

\begin{enumerate}
\item $G(\Z)^+\cdot S_G = \Hh_G$
\item For an irreducible algebraic curve $C\subset \Hh_G$ of degree
$k$, we have the bound
\[\int_{C\cap S_G} dC \ll k\]
where the implied constant depends on $G$
\item The projection map $\pi_g:S_G\rightarrow \A_g$ is definable.
\end{enumerate}
\end{thm}

\begin{proof} Assertion (1) was already proven in the discussion above.
For (2) we need only observe that if $C$ is algebraic of degree $k$,
then so is  $\gamma\cdot C$. Hence
\[\int_{C\cap S_G} dC \leq \sum_{i=1}^m\sum_{j=1}^n\int_{C\cap \beta_j\gamma_iS_g} dC=
\sum_{i=1}^m\sum_{j=1}^n\int_{(\beta_j\gamma_i)^{-1}C\cap S_g} d(\beta_j\gamma_i)^{-1}C\ll k\]
where the last inequality follows from Theorem \ref{SmallCurves}.
Finally, (3) follows easily from the fact that $\pi_g:S_g\rightarrow\A_g$
is definable, as per Theorem \ref{define}.
\end{proof}

\bigskip

\section{Volumes of algebraic curves near the boundary}

Take $C\subset M_g(\C)^{\rm sym}$ to be an irreducible
algebraic curve of
degree $k$ in $\Hh_g$. Then its Zariski closure $C^{\rm zar}$
must intersect the boundary $\partial\Hh_g$ in some
real algebraic curve $C_0$; to see that $C^{\rm zar}$ cannot be
contained in $\Hh_g$ recall that there is a birational
algebraic map taking
$\Hh_g$ to a bounded set.
Our goal in this section is to show that the volume of $C$
near $C_0$ is large.
This idea and its execution are due to Ullmo and Yafaev
\cite{UYHYPERBOLIC}.

\begin{thm}\label{boundarybound}
For  $M>1$ set $C_M:=\{Z\in C\mid h(Z)\leq M\}$.
Then \[M\prec\int_{C_M}dC.\]
\end{thm}

\begin{proof}

We proceed as in \cite{UYHYPERBOLIC}.
Pick a smooth compact piece $I\subset C_0$, and a point $p\in I$.
For $0<\alpha<\beta< 2\pi$ set
$$\Delta_{\alpha, \beta}:=\{z=re^{i\theta}\mid 0\leq r<1,\alpha\leq \theta
\leq\beta\}$$
and
$$C_{\alpha, \beta}:=\{z=e^{i\theta}\mid \alpha\leq \theta\leq\beta\}.$$

We may find $\alpha, \beta$ and a real analytic map
$$\psi:\Delta_{\alpha,\beta}\rightarrow C$$
which extends to a real analytic function from a  neighbourhood of
$\Delta_{\alpha, \beta}\cup C_{\alpha, \beta}$ to $C^{\rm zar}$
such that $\psi(C_{\alpha,\beta})\subset \partial\Hh_g$.
Composing with $|Y|$
gives a real analytic function on a  neighbourhood of
$\Delta_{\alpha,\beta}\cup C_{\alpha, \beta}$ which is positive on
$\Delta_    {\alpha, \beta}$ and vanishes exactly when
$1-z\bar{z}$ vanishes. Thus, there exists $\lambda>0$ such that
$|Y|^2 = (1-z\bar{z})^{\lambda}\cdot\psi_1(z)$ where $\psi_1(z)$ is a
real analytic function that is positive on $\Delta_{\alpha,\beta}$
and which does
not vanish identically on $C_{\alpha,\beta}$. Thus by changing $\alpha$
and $\beta$ if necessary we can ensure that $\psi_1(z)$ is
non-vanishing on $C_{\alpha,\beta}$, so that
\begin{equation}\label{boundary1}
\log|Y|= \frac{\lambda}{2}\log(1-z\bar{z}) + O(1).
\end{equation}

Now, as in Ullmo-Yafaev we can also ensure that if
$\omega$ denotes the
K\"ahler form  on $\Hh_g$ and
$\omega_\Delta=i{dz\wedge d\bar{z}}/{(1-|z|^2)^{2}}$
then
\begin{equation}\label{boundary2}
\psi^*(\omega)=s\omega_{\Delta}+\eta
\end{equation}
where $\eta$ is smooth in some neighbourhood of $C_{\alpha,\beta}$,
and $s$ is some positive integer.
Finally, for $\delta<1$ we set
$I_{\delta}=\Delta_{\alpha,\beta}\cap\{|z|<1-\delta\}$.

A computation gives we have
$$\int_{I_{\delta}} \omega_{\Delta} \gg \frac{1}{\delta}.$$
Combining this with equations \eqref{boundary1} and \eqref{boundary2}
gives the result.
\end{proof}

\section{Proof of Ax-Lindemann}

We can now prove the Ax-Lindemann
(or Ax-Lindemann-Weierstrass) theorem for $\A_g$.
We give a more general formulation applicable to semi-algebraic
subsets $W\subset\pi_g^{-1}(V)$. Such a set $W$
will be called {\it irreducible\/}
if it is not the union of two non-empty relatively closed
proper subsets in the topology
induced on it by the Zariski topology of algebraic sets defined over $\R$
(see \cite{GV}).

\begin{thm}\label{ALW}
Let $V\subset\A_g$ be an irreducible algebraic variety, and
suppose that $W\subset\pi_g^{-1}(V)$ is a connected irreducible
semi-algebraic subset of $\Hh_g$. Then there exists a weakly special
subvariety $S\subset V$ such that $W\subset\pi_g^{-1}(S).$
\end{thm}

\begin{proof}

We proceed by induction on the dimension $V$,
the case of $\dim(V)=0$ being
obvious, as all points are weakly special;
for the same reason we may assume $W$ has positive dimension.
We may assume without loss of generality that $W$ is maximal
(i.e. if $W'$ is semi-algebraic with $W\subset W'\subset \pi_g^{-1}(V)$
then $W'$ has $W$ as a component).
Likewise, we may
assume by our induction hypothesis that $V$ is minimal in    the sense
that there does not exist
$V'\subset V$ with $\dim V'<\dim V$ and $W\subset\pi_g^{-1}(V')$.
Note that since $W$ is maximal, by
Lemma 4.1 of \cite{PT}, $W$ must be a complex algebraic subvariety.
To prove the theorem we must show that $V$ is weakly special,
so that $W$ is an irreducible component of $\pi_g^{-1}(V)$, by
Lemma \ref{UYS}.

Take now $S_0$ to be the minimal weakly special
subvariety of $\A_g$ containing $V$ and write $U_0$
for the irreducible component of  $\pi_g^{-1}(S_0)$
containing $W$ (so we have to
show that $S_0=V$ and $W=U_0$).

Let $\Gamma_V\subset\Sp_{2g}(\Z)$ be the monodromy group of $V$,
and let $G_0$ be the connected
component of the Zariski closure of $\Gamma_V$.
Then by Lemma \ref{MinimalWS},
$G_0$ is a semi-simple
group over $\Q$, and for any point $u_0\in U_0$ we have
$U_0=G_0(\R)^+\cdot u_0$. Assume without loss of generality
that $W\cap S_{G_0}\neq \emptyset$.

Write $Y$ for a connected component of $\pi_g^{-1}(V)$ which
intersects the open part of $S_{G_0}$, and set $Y^0=Y\cap S_{G_0}$.
Note that $Y^0$ is definable (in $\R_{\rm an, \exp}$; \cite{PEST}). Write
$$X=\{\gamma\in G_0(\R)\mid \dim \big(\gamma\cdot W\cap Y^0\big)
=\dim W\},$$
which likewise is definable. Observe that $\gamma\cdot W\subset Y$
for all $\gamma\in X$ by the dimension assumptions
and analytic continuation.

For a set $A\subset \Sp_{2g}(\R)$ and a real number $T\ge 1$
define the counting function
$$
N(A,T)= \#\{M\in A\cap \Sp_{2g}(\Z): H(M)\le T\}.
$$
Here $H(M)$ is the multiplicative Weil height of $M$, i.e. the
maximum size of its integer entries, so that $H(M)=h(M)$ as
defined in 2.6 (as not all its entries can vanish).

\begin{lemma}\label{manydomains}
We have $T\prec N(X,T)$, implied constants depending on $X$.
\end{lemma}

\begin{proof}
Let $C\subset W$ be an algebraic curve. For $T>0$, define
$$C_T = \{Z\in C\mid h(Z)\leq T\}, \quad
X_T=\{\gamma\in X\mid h(\gamma)\leq T\}.$$
Since $S_{G_0}$ contains a fundamental domain,
for each point $Z\in C$ there exists a $\gamma\in\Gamma_V$ such that
$\gamma\cdot Z\in Y^0$. By Lemma \ref{movingbound} there exists
$M>0$ such that $H(\gamma)\leq h(Z)^M$ for sufficiently large $h(Z)$.
Hence for $T\gg 1$, we must have
$$
C_{T^{\frac1M}} \subset
\bigcup_{\gamma\in X_{T}\cap\Gamma_V} \gamma^{-1}S_{G_0}.$$
In particular, we have that
$$\textrm{Vol}(C_{T^{\frac1M}}) \leq
\sum_{\gamma\in X_{T}\cap\Gamma_V}
{\rm Vol}(\gamma\cdot C\cap S_{G_0}).
$$
Combining Theorems \ref{boundarybound} and \ref{SmallCurvesGen}
now gives the result.
\end{proof}

Since $N(X,T)$ grows faster than some positive power of $T$,
the Counting Theorem  (in the form \cite{PILAOAO}, Theorem 3.6)
implies that there are semi-algebraic varieties $W_1\subset X$ of
positive dimension containing arbitrarily many points
$\gamma\in\Gamma_V$
such that $W_1\cdot W\subset Y$. To see this, supposing
$N(X,T)\gg T^\eta$ with $\eta>0$, apply \cite{PILAOAO}, Theorem 3.6
with $\mu=0, k=1$, and $0<\epsilon<\eta$. As the $\gg T^\eta$ rational
points up to height $T$ are contained in $\ll T^\epsilon$ blocks
(as defined there) provided
by the theorem, there must be a block containing
$\gg T^{\eta-\epsilon}$ rational
points. Note that the simpler version \cite{PILAOAO}, Theorem 3.2
(which is \cite{PW}, Theorem 1.9 in the case $k=1$)
does not give this conclusion.
For such $W_1$ and $\gamma\in W_1\cap\Gamma_V$ we have
$\gamma^{-1}W_1\cdot W$ contains $W$, and so by our
maximality assumption we have $$\gamma^{-1}W_1\cdot W = W.$$

Now let $\Theta$ be the algebraic group over $\R$ which is the
Zariski closure in $G_{0,\R}$ of the subset
$$
\{\gamma\in G_0(\R)\mid \gamma\cdot W=W\}.
$$
Let $\Theta^0$ be its connected component. Since
$\gamma^{-1}W_1\subset \Theta$, it follows that $\Theta^0$
has positive dimension.
Let $H$ be the maximal connected  algebraic
subgroup of $G_0$ defined over $\Q$ such that
$H(\R)\subset\Theta^0(\R)$. We know that  $\Theta^0$ has
infinitely many rational (in fact integral) points, hence so does
$H(\R)$, and so $\dim(H(\R))>0$.

Suppose that $Y$ is NOT invariant under $H(\R)^+$.
Since $H(\Q)$ is dense in $H(\R)$ (see e.g. \cite{SANSUC}, Cor. 3.5)
we can find an element
$h\in H(\Q)^+$ such that $Y$ is not invariant under $h$.
Let $Y'=Y\cap hY$. Then $\pi_g(Y')$ is a closed algebraic
proper subvariety $V'\subset V$ with and $W\subset Y'$
(in fact $V'$ is a component of the intersection of
$V$ with one of its Hecke translates).
This contradicts our minimality assumption on $V$,
hence $Y$ is invariant under $H(\R)^+$.

Let $H'$ be the smallest algebraic subgroup of $G_0$ containing
the conjugates of $H$ by all $\gamma\in\Gamma_V$.
Since $Y$ is invariant under $H(\R)^+$ and under $\Gamma_V$ it is
invariant under $H'(\R)^+$. Now $H'$ is an algebraic group
that is invariant by conjugation under $\Gamma_V$,
and hence also under $G_0$, the Zariski closure of $\Gamma_V$.
Hence $H'$ is normal in $G_0$.
Note that $H'(\R)^+\cdot W\subset Y$ so by
maximality of $W$ we conclude that $H'(\R)^+\cdot W=W$,
so that $H'=H$.

Next consider the map $\phi:G_0\rightarrow G_0^{\rm ad}$, where
$G_0^{\rm ad}$ is the adjoint form of $G_0$. We can therefore write
$G_0^{\rm ad}=\prod_{i=1}^r G_i$ where the $G_i$ are $\Q$-simple
algebraic groups. Therefore there is some non-empty subset
$I\subset\{1,\dots,r\}$ such that $$\phi(H)=\prod_{i\in I} G_i.$$

We write $U_0\cong\prod_i U_i$ where the $U_i$ are
hermitian symmetric
spaces associated to $G_i$. Thus $W$ can be written as
$\prod_{i\in I}U_i\times W'$, where $W'$ is an irreducible algebraic
subvariety of $\prod_{i\in I^c}U_i$.
If $W'$ is a point then $\pi_g(W)$ is
weakly special, and so we must have $W=U_0$ as desired. Hence, we assume from now on that
$\dim(W')>0$.

Now pick any point $u_I\in\prod_{i\in I} U_i$ and consider the variety
$$W^{\rm new}= u_I\times W'.$$
Moreover, let
$$V^{\rm new} = V\cap \pi_g(u_I\times\prod_{j\in I^c}U_j),$$
which is evidently an algebraic variety as $\prod_{j\in I^c}U_j$ is
weakly special.
Then $$\pi_g(W^{\rm new})\subset V^{\rm new}$$ and so by our
induction hypothesis there is a weakly special variety
$$S^{\rm new}\subset u_I\times \prod_{j\in I^c}U_j$$ such that
$\pi_g(S^{\rm new})\subset V^{\rm new}$ and
$W^{\rm new}\subset S^{\rm new}$.
By projection, we can identify $S^{\rm new}$ with its image in
$\prod_{j\in I^c}U_j$.
Since weakly special varieties are closed under intersection
(Corollary  \ref{INT}) we can consider
$S_{\rm min}\subset\prod_{j\in I^c}U_j$
to be the minimal weakly special variety containing $W'$.
By the discussion above,
for each point $u_I\in\prod_{i\in I}U_i$ we know that
$\pi_g(u_I\times S_{\rm min})\subset V$. In other words, the variety
$$S=\prod_{i\in I} U_i\times S_{\rm min}$$ is weakly special and
satisfies both $W\subset S$ and $\pi_g(S)\subset V$. By our maximality assumption,
we conclude that $W=S=U_0$ and $V=S_0$ as desired.
%
%
%
\end{proof}

\bigbreak

\section{Application to Andre-Oort}

In this section we give the application to the Andre-Oort conjecture.
As already mentioned, the size of the Galois orbit of a special point plays
a crucial role.
For a point $x\in\A_{g}$ let $A_x$ denote the corresponding
$g$-dimensional
principally polarized abelian variety,
$R_x=Z({\rm End}(A_x))$ the centre of the endomorphism ring of $A_x$,
and ${\rm Disc\/}(R_x)$ its discriminant
(for further details see \cite{JT}).
In general one expects
to have the following lower bound suggested by Edixhoven in
\cite{EMO}.

\begin{conj}
Let $g\ge 1$. Then, for a special point $x\in\A_{g}$,
$$
|{\rm Disc\/}(R_x)| \prec |{\rm Gal}(\overline{\Q}/\Q)\cdot x|\
$$
(with the implied constants depending on $g$).
\end{conj}

\noindent

For $g=1$ this is known by the theory of complex multiplication
of elliptic curves (see e.g. \cite{CM}) and the Landau-Siegel
lower bound for the class number of an imaginary quadratic field
\cite{LANDAU, SIEGELAA}.
As already mentioned, the second author \cite{JT}
has affirmed this conjecture

(1) for $g\le 6$, and

(2) for all $g$ under GRH for CM fields (see also Ullmo-Yafaev
\cite{UYGB}).

\noindent
All the unconditional lower bounds are ineffective;
hence Theorem 1.2,
which we now establish by proving Theorem 7.1 below, is ineffective
(other aspects of the proof are also ineffective,
including the Counting Theorem, though this ineffectivity
seems less serious).
The only non-trivial case of the Andre-Oort conjecture known
unconditionally and effectively
is that of the product of two modular curves, due independently
to Kuhne \cite{KUHNE} and Bilu-Masser-Zannier \cite{BILUMZ}
(Andr\'e's original proof \cite{AA} is ineffective).

\begin{thm}
Suppose Conjecture 7.1 holds for $g$.
Let $V$ be an irreducible closed algebraic subvariety of $\A_g$.
Then $V$ contains only finitely many
maximal special subvarieties.
\end{thm}

Before proving this theorem we need some lemmas, in which we take
$V\subset\A_g$ to be an irreducible closed algebraic subvariety.
We do not for the moment assume Conjecture 7.1.

We set $G\subset {\rm GSp}_{2g, \Q}$ to be the
generic Mumford-Tate group of an irreducible component $Z$ of
$\pi_g^{-1}(V)$. Then, for any $z\in Z$,  $\pi_g(G(\R)^+\cdot z)$ 
is the smallest special subvariety containing $V$.
 We may assume that $V$ contains special points, so that this is also the smallest such weakly special subvariety.
Let $S_G$ be as defined above Theorem 4.4, and recall it contains
a fundamental domain for $G(\Z)^+$.
Let $Y^0=\pi_g^{-1}(V)\cap S_G$,
which, as we have already observed in Theorem 4.4, is a definable set.

We consider orbits $H(\R)^+\cdot z, z\in\Hh_g$ that lie inside $\pi_g^{-1}(V)$,
where $H\subset{\rm Sp}_{2g, \R}$ is a semi-simple group.
As such orbits are real algebraic, we observe that if
$$
\dim H(\R)^+\cdot z=
\dim\left( H(\R)^+\cdot z \cap\left(\pi_g^{-1}(V)\cap S_G\right)\right)
$$
then, by analytic continuation, we have
$H(\R)^+\cdot z\subset\pi_g^{-1}(V)$.
We call such an orbit {\it maximal\/} if it is not contained in an orbit
$K(\R)^+\cdot z$, where
$K\subset{\rm Sp}_{2g, \R}$ is semi-simple, $K(\R)^+\cdot z$ is
contained in $\pi_g^{-1}(V)$ and has larger dimension than
$H(\R)^+\cdot z$.
By Ax-Lindemann (\ref{ALW}), every maximal orbit is weakly-special,
and Lemma \ref{WeakChar} implies that there is a connected,
semi-simple $\Q$-subgroup $F\subset G$ with no compact $\Q$-factors,
such that $F(\R)^+\cdot z=H(\R)^+\cdot z$.

\begin{definition*} We set $\mathcal{C}$ to be the set of all connected semi-simple $\Q$-subgroups $F\subset G$ which have no
compact $\Q$-factors, and such that there exists $y\in\Hh_g$ for which $F(\R)^+\cdot y$ is weakly special.
 \end{definition*}
 
 For $F\in\mathcal{C}$ and for every weakly special subvariety $F(\R)^+\cdot y\subset\Hh_g$ there
is a Shimura subdatum $(H,X_H)\subset (G,X_G)$ and a splitting
$$
(H^{\rm ad}, X^+_{H^{\rm ad}}, \Gamma)=
(H_1, X_1^+, \Gamma_1)\times(H_2, X_2^+, \Gamma_2),
$$
and a point
$h_2\in X^+_2$
such that $X^+_1\times\{h_2\}=F(\R)^+\cdot y$.

Consider the group
$$\Gamma':=\{\gamma\in G(\Z)^+\mid \gamma\cdot F(\R)^+\cdot y=
F(\R)^+\cdot y\}.$$
This group has the property that
$\Gamma'\backslash F(\R)^+\cdot y$ is finite volume.
However, letting $\widetilde{H_1}$ denote the pullback of $H_1$ to $H$, the subgroups $F(\Z)^+$ and $\widetilde{H_1}(\Z)^+$ also have this
property, and hence their images in $G^{\rm ad}$ are commensurable.
Since neither $F$ nor $\widetilde{H_1}$ have compact $\Q$-factors
and both are connected,
$F\cdot Z(\widetilde{H_1})=\widetilde{H_1}$
by \cite{PR}, Theorem 4.10. Letting $N(F)$ denote the
connected component of the normalizer of $F$ in $G$, it follows that $H\subset N(F).$

\begin{lemma}
Let $F\in\mathcal{C}$.
Then $N(F)$ is a reductive $\Q$-group.

\end{lemma}

\begin{proof}

Let $y\in\Hh_g$ be such that $F(\R)^+\cdot y$ is a weakly special
subvariety. Recall that $y$ gives a homomorphism
$y:\DT\rightarrow G_{\R}$, and conjugation
by $y(i)$ gives a Cartan involution of $G_{\R}^{\rm ad}$
which preserves the image of $F _{\R}$ in $G_{\R}^{\rm ad}$.
Therefore, it must also preserve $N(F)_{\R}$.
By \cite{SAT} \S I, 4.3, $N(F)$ is reductive. The fact that it is
defined over $\Q$ follows easily from the fact that $F$ is.
\end{proof}

Write ${\rm Com}(N(F))$ for the product of all compact
$\Q$-factors of $N(F)^{\rm ad}$
and let $N(F)_{\rm sh}$ denote the identity
component of the kernel of the
projection $N(F)\rightarrow {\rm Com}(N(F))$.
Then $y:\DT\rightarrow G_{\R}$ factors through
$\phi(H_1\times H_2)$
and hence also through $N(F)_{\rm sh}$. Thus
$N(F)_{\rm sh}(\R)^+\cdot y$
defines a special subvariety of $G(\R)^+\cdot y$.

\begin{lemma}\label{SpFam}
Let $F\in\mathcal{C}$.
Then the set of special subvarieties of $\Hh_g$ of the form
$N(F)_{\rm sh}(\R)^+\cdot y$ such that
$y(\DT)\subset N(F)_{\rm sh,\R}$ is finite as $y$ varies over $\Hh_g$.

\end{lemma}

\begin{proof}

The requirement is that $y:\DT\rightarrow G_{\R}$ has image in
$N(F)_{\rm sh,\R}$,
and we would like to show that the number of such images that
arise is finite up to $N(F)_{\rm sh}(\R)^+$ conjugacy.
Since the homomorphisms we are interested in correspond to
Shimura varieties,
it means each such homomorphism also defines a
hermitian symmetric domain,
which has as its group of biholomorphisms the
corresponding adjoint real group
$N(F)_{\rm sh,\R}^{\rm ad}$. It is well known that the number of
such domains of a bounded dimension is finite,
hence there are finitely many types
$(N(F)_{\rm sh,\R}^{\rm ad},y:\DT \rightarrow N(F)_{\rm sh,\R}^{\rm ad})$
corresponding to a hermitian
symmetric space
up to an automorphism of
$N(F)_{\rm sh,\R}^{\rm ad}$. Since $N(F)_{\rm sh}^{\rm ad}$
is semi-simple,
its outer automorphism group is finite.
This means that, up to $N(F)_{\rm sh}(\R)^+$ conjugacy,
there are only finitely many
homomorphisms $y:\DT\rightarrow N(F)^{\rm ad}_{\rm sh,\R}$ that arise
from points $y\in\Hh_g$.
But recall that we know that the weights
of the action ${\rm ad}\circ y$ of $\DT$ on the complexified Lie algebra of
$\Sp_{2g}$ are restricted to a finite set. This implies that for each
homomorphism to $N(F)_{\rm sh, \R}^{\rm ad}$,
only finitely many possible lifts to
$N(F)_{\rm sh,\R}$ actually occur in the $G(\R)^+$-conjugacy class.
This completes the proof of the assertion.
\end{proof}

For the proof of Theorem 7.1, it will be useful to make the following definition:

\begin{definition*} For each $F\in \mathcal{C}$,
define $S(F)$ to be the
union of $\pi_g(F(\R)^+\cdot z)$ over the
maximal $F$-orbits $F(\R)^+\cdot z$, which are also
special subvarieties and satisfy
$$
\dim F(\R)^+\cdot z=
\dim\left( F(\R)^+\cdot z \cap\left(\pi_g^{-1}(V)\cap S_G\right)\right).
$$
We let $W(F)$ be the corresponding union over maximal $F$-orbits
which are weakly special.
\end{definition*}

Finally, we need the following result,
which is well known to the experts but for which we
know of no easy reference in the literature:

\begin{lemma}\label{finitedisc}
For any $B>0$, there are only finitely many special points
$x\in\A_{g}$ such that $|{\rm Disc\/}(R_x)|<B$.
\end{lemma}

\begin{proof}

Let $x$ be such a point and consider its lift $Z$ to the fundamental
domain $S_g$. Then by Theorem 3.1 in \cite{PT}
we know that $H(Z)\prec |{\rm Disc\/}(R_x)|$, where $H(Z)$
is the height of the point $Z$. Since all the co-ordinates
of $Z$ are algebraic of degree at most $2g$,
there are only finitely many such $Z$ by Northcott's Theorem.
This completes the proof.
\end{proof}

We can now prove Theorem 7.1.

\begin{proof} We now assume that Conjecture 7.1 holds for $g$.
The proof is by induction on $\dim V$, therefore we may assume
that $V$ contains a Zariski-dense set of special points,
and we must then show that $V$ is a special subvariety.
Since special points are defined over $\Qa$, we must have that
$V$ is defined over a field $K$ of finite degree over $\Q$.

Let $x$ be a special point in $V$, and let $Z$ be a preimage of $x$
contained in the standard fundamental domain
$S_g$ of $\Sp_{2g}(\Z)$ as in \cite{PT}.
By Theorem 3.1 in \cite{PT}, we know that the height $H(Z)$ of $Z$
satisfies $H(Z)\prec|\Disc(R_x)|$.
The conjugates $x'$ of $x$ over $K$ again lie on $V$, are special points
of the same discriminant, and have pre-images $Z'$ in $S_g$.
Thus by Conjecture 7.1, if $|{\rm Disc\/}(R_x)|$ is sufficiently
large, the Counting Theorem (in the form \cite{PILAOAO}, Theorem 3.2)
implies that ``many'' of the $x'$  (i.e. a positive power of
$|\Disc(R_x)|$ of them) are contained in the
algebraic part (as defined there) of $Y^0$, which thus contains a
semi-algebraic set of positive dimension containing at least
one (in fact many, if we use \cite{PW}, 3.6, but one will suffice)
preimage $Z'$ of a Galois conjugate of $x$ over $K$.
By Ax-Lindemann (the semi-algebraic version: Theorem \ref{ALW})
$Z'$ is contained in a weakly special subvariety
$S\subset \pi^{-1}(V)$ which we may choose
to be maximal, and is in fact special as $x$ is special.
This implies that $\pi_g(S)$ is algebraic,
and hence the union of all Galois conjugates of $\pi_g(S)$
over $K$ contains $x$ and lies inside $V$. Hence, by Lemma
\ref{finitedisc}, all but finitely many special points $x\in V$
are contained in a positive dimensional special subvariety of $V$.

It therefore suffices to establish that there are only finitely many
positive dimensional maximal special subvarieties of $V$,
for then the Zariski-density of special points will imply that
$V$ itself is special.

By Lemma \ref{WeakChar}, any such special subvariety is the image under $\pi_g$ of an orbit of
$F(\R)^+$ for some $F\in \mathcal{C}$.
We first show that, for each $F\in \mathcal{C}$,
the set $S(F)$ is a finite union of special subvarieties.

We first consider the case
where $N(F)_{\rm sh}$ is not equal to all of $G$.
Then by Lemma \ref{SpFam}, special subvarieties of $V$ which are
images under $\pi_g$ of orbits of $F(\R)^+$ lie in the union of finitely many special subvarieties whose generic Mumford-Tate group $M$
satisfies $M^{\rm der}=N(F)_{\rm sh}^{\rm der}$. But the smallest
special subvariety containing $V$
has generic Mumford-Tate group $G$, and so the intersections
of these special subvarieties with $V$ must all be proper.
Hence we are done in this case by induction on the dimension of $V$.

Now consider the case that $N(F)_{\rm sh}$ is equal to $G$.
This means that $F$
is normal in $G$, and so $G$ splits as an almost direct product
$G=F\cdot Z_G(F)$.
Let $y$ be a point in $\pi_g^{-1}(V)\cap S_G$, so that $G(\R)^+\cdot y$
is a hermitian symmetric space with $\pi_g(G(\R)^+\cdot y)$ the
minimal special subvariety containing $V$.
Then one has a splitting of the hermitian symmetric space
$G(\R)^+\cdot y$ as
$F(\R)^+\cdot y\times Z_G(F)(\R)^+\cdot y$.

Let $z_{CM}\in S_G$ be a special point. Then
$F(\R)^+\cdot y\cap Z_G(F)(\R)^+\cdot z_{CM}$ consists of a
single point, and moreover
this point is special if and only if $F(\R)^+\cdot y$ is a
special subvariety, if and only if
$F(\R)^+\cdot y$ contains special points
(see \cite{MOONEN}, Theorem 4.3).
Thus special subvarieties of the form $F(\R)^+\cdot y$ are in bijection
with special points on $Z_G(F)(\R)^+\cdot z_{CM}$.
Moreover, if we let $X_F^+, X_G^+$, and $X_{Z_G(F)}^+$ denote the
hermitian symmetric spaces  corresponding respectively to $F,G$
and $Z_G(F)$, then we get a finite
algebraic map of Shimura varieties
$$\psi:F(\Z)^+\backslash X_F^+\times Z_G(F)^+(\Z)\backslash
X_{Z_G(F)}^+\rightarrow G(\Z)^+\backslash X_G^+.$$
Put $W= \pi_g(Z_G(F)(\R)^+\cdot z_{CM})$ and let
$$
V'=\{x\in W\mid x\in V\textrm{ \& }
F(\Z)^+\backslash X_{F}^+\times{\rm pr}_2\psi^{-1}\{x\}\subset
\psi^{-1}(V)\}.
$$
Then $V'$ is a closed algebraic subvariety of $\A_g$ of dimension
lower than $V$ (as the orbits of $F(\R)^+$ have positive dimension),
and special points in $V'$ correspond
in a finite-to-one fashion to special subvarieties of the form
$\pi_g(F(\R)^+\cdot y)$ in $V$.

By induction on dimension the conclusion of 7.1 is true for
$V'$ and thus the set of special points in $V'$
is contained in a finite union of special subvarieties of $V'$.
Hence, the images of \emph{maximal} orbits $\pi_g(F(\R)^+\cdot y)$
correspond to the finite set of isolated special points in $V'$.
Thus we have finiteness also in this case.

\begin{lemma}\label{CountVar}
The union of all positive dimensional
weakly special subvarieties of $V$ is a countable union of algebraic varieties.
\end{lemma}

\begin{proof}
Since there are only countably many $\Q$-subgroups of $G$, 
we need only consider those weakly special subvarieties arising 
from a single Shimura subdatum $(H,X_H)\subset (G,X_G)$ 
and a splitting 
$$(H^{\rm ad},X^{\rm ad,+}_H) = (H_1,X_1^+)\times(H_2,X_2^+).
$$ 
We then get an algebraic
morphism 
$$
\Gamma_1\backslash X_1^+ \times 
\Gamma_2\backslash X_2^+\rightarrow G(\Z)^+\backslash X_G^+
$$ 
where $\Gamma_1\subset H_1$ and $\Gamma_2\subset H_2$ 
are appropriate arithmetic subgroups.
Then the set of points $y\in\Gamma_2\backslash X_2^+$ 
such that $X_1^+\times\{y\}$ maps inside of $V$ is clearly 
an algebraic variety. This completes the proof.
\end{proof}

From Lemma \ref{CountVar} it follows that   $\bigcup_{F\in\mathcal{C}}W(F)$ is a
countable union of  algebraic varieties.
We state two further lemmas (which do not require the assumption of
Conjecture 7.1).

\begin{lemma}\label{defweak}
The union of all weakly special subvarieties of $V$ of positive
dimension is a definable subset of $\A_g(\C)$.
\end{lemma}

\begin{proof}

Let $T\subset V$ be weakly special.
Since $S_G$ contains a fundamental domain for the
action of $G(\Z)$, every point of $T$
is contained in an image of the form $\pi_g(H(\R)^+\cdot z\cap S_G)$
for $H$ a semi-simple subgroup of $G$ and $z\in S_G$.
There are only finitely many semi-simple real groups that embed
into ${\rm Sp}_{2g,\R}$, and the embeddings come in finitely many
families
up to conjugacy by (\cite{EMV}, A.1).
We restrict for the moment to a particular
semi-simple subgroup $H\subset G_{\R}$. For each $y\in\ G(\R)$ let
$H_y:=yHy^{-1}$. Since the dimension of orbits
is definable, it follows that the set
$$
T_H:=\{(y,z)\in \Sp_{2g}(\R)\times S_G:
H_y(\R)^+\cdot z{\rm\ is\ maximal}\}
$$
is definable (the maximality is definable as there are only finitely many
conjugacy classes of groups to compare with),
and the union of the
$$
\pi_g(H_y(\R)^+\cdot z\cap S_G)
$$
over $(y,z)\in T_H$ is thus also definable. Taking the union over
the finitely many conjugacy
classes of semi-simple groups $H$ proves the result.
\end{proof}

\begin{lemma}
Let $A$ be a quasiprojective complex algebraic variety.
Suppose $V_i$ is a countable sequence of closed complex algebraic
subvarieties of $A$ whose union is definable. Then there is a natural
number $m$ such that
$$
\bigcup_{i\in\mathbb{N}} V_i = V_1\cup V_2\cup\cdots \cup V_m.
$$
\end{lemma}

\begin{proof}

Let $X=\bigcup_{i\in\mathbb{N}} V_i$. Since $X$ is definable
(in $\R_{\rm an\ exp}$)
it has an analytic cell decomposition (\cite{DM}),
$X=\cup_{j} U_j$. For each $j$, there is
an $i$ such that $U_j\cap V_i$ has the same dimension as $U_j$.
But since $U_j$ is
a connected real analytic variety, if a polynomial vanishes
on an open set of $U_j$
it must vanish on $U_j$ identically. Hence $U_j\subset V_i$.
\end{proof}

By applying the preceding lemmas we see that there is a finite set
$\mathcal{V}\subset\mathcal{C}$ such that
the union of all maximal weakly special subvarieties of $V$
of positive dimension is
$\bigcup_{F\in\mathcal{V}} W(F).$
Consider a positive dimensional special subvariety $T\subset V$.
We have $T\subset \bigcup_{F\in\mathcal{V}} W(F)$. Now special
points are Zariski-dense in $T$, but the special points in each $W(F)$
are contained in $S(F)$. Thus
$T\subset \bigcup_{F\in\mathcal{V}} S(F)$. We conclude that
all maximal special subvarieties of $V$ of positive dimension
are contained in the finite union $\bigcup_{F\in\mathcal{V}} S(F)$.
Hence the theorem is true for $V$.
\end{proof}
\bigskip
\bigbreak

\bigskip
\bigskip

\bigskip
\bigskip

\end{document}